\documentclass[12pt, reqno]{amsart}

\usepackage{amsmath,amsthm,amssymb}
\usepackage{amstext,amscd,amsfonts,amsbsy,amsxtra,latexsym,mathtools}
\newcommand{\C}{\mathbb C}
\renewcommand{\H}{\mathbb H}
\newcommand{\Z}{\mathbb Z}
\newcommand{\N}{\mathbb N}
\newcommand{\abcd}{\left(\begin{smallmatrix}a&b\\c&d\end{smallmatrix}\right)}
\DeclareMathOperator{\SL}{SL}
\newcommand{\eps}{\varepsilon}
\newcommand{\nequiv}{\not\equiv}
\renewcommand{\(}{\left(}
\renewcommand{\)}{\right)}
\allowdisplaybreaks

\usepackage{hyperref}
\hypersetup{
  colorlinks   = true, 
  urlcolor     = blue, 
  linkcolor    = blue, 
  citecolor    = red   
}

\newtheorem{theorem}{Theorem}[section]
\newtheorem*{cor*}{Corollary}
\newtheorem{lemma}[theorem]{Lemma}

\newtheorem{proposition}[theorem]{Proposition}
\newtheorem{corollary}[theorem]{Corollary}

\theoremstyle{remark}

\newtheorem*{remark}{Remark}

\numberwithin{equation}{section}

\begin{document}

\title[Jacobi forms and generalized Frobenius partitions]{Formulas for Jacobi forms and generalized Frobenius partitions}
\author[Bringmann, Rolen, Woodbury]{Kathrin Bringmann, Larry Rolen, Michael Woodbury}
\address{Mathematical Institute\\University of
Cologne\\ Weyertal 86-90 \\ 50931 Cologne \\Germany}
\email{kbringma@math.uni-koeln.de} 
\address{
Hamilton Mathematics Institute
\&
School of Mathematics \\
Trinity College \\
Dublin 2, Ireland}
\email{lrolen@tcd.maths.ie}
\address{Mathematical Institute\\University of
Cologne\\ Weyertal 86-90 \\ 50931 Cologne \\Germany}
\email{woodbury@math.uni-koeln.de}
\date{\today}
\thanks{The research of the first author is supported by the Alfried Krupp 
Prize for Young University Teachers of
the Krupp foundation and the research leading to these results 
receives funding from the European Research
Council under the European Unionâ's Seventh Framework Programme 
(FP/2007-2013) / ERC Grant agreement n.
335220 - AQSER}
\begin{abstract}
Since their introduction by Andrews, generalized Frobenius partitions have interested a number of authors, many of whom have worked out explicit formulas for their generating functions in specific cases. This has uncovered interesting combinatorial structure and led to proofs of a number of congruences.  In this paper, we show how Andrews' generating functions can be cast in the framework of Eichler and Zagier's Jacobi forms.  This reformulation allows us to compute explicit formulas for the generalized Frobenius partition generating functions (and in fact provides formulas for further functions of potential combinatorial interest), and it leads to a recursion formula to calculate them in terms of infinite $q$-products.  
\end{abstract}
\maketitle

\section{Introduction and statement of results}

Given a \emph{partition} of $n$, that is, a sequence of non-increasing positive integers summing to $n$, it is often useful to consider its associated Frobenius coordinates.  This is done by reading off the leg and arm lengths along the diagonal of its Ferrers diagram. This gives a bijection between partitions of $n$ and arrays
\begin{align}\label{eq:array}
 \left( \begin{array}{cccc}
    a_1 & a_2 & \cdots & a_m \\
    b_1 & b_2 & \cdots & b_m \end{array} \right)
\end{align}
with integral coordinates $a_1>a_2>\text{\dots}>a_m\geq 0$ and $b_1>b_2>\text{\dots}>b_m\geq 0$ such that $n=m+\sum_{j=1}^m(a_j+b_j)$. For example, associated to the partition $\lambda=(5,4,4,2,1)$ of $16$, we have the following diagram, with the boxes along the diagonal labeled with the symbol $\bullet$: 

\begin{figure}[h]
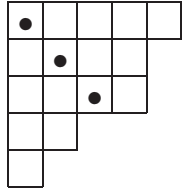

{\def\lr#1{\multicolumn{1}{|@{\hspace{.6ex}}c@{\hspace{.6ex}}|}{\raisebox{-.3ex}{$#1$}}}
\raisebox{-.6ex}{$\begin{array}[b]{*{5}c}\cline{1-5}
\lr{\bullet}&\lr{\phantom{x}}&\lr{\phantom{x}}&\lr{\phantom{x}}&\lr{\phantom{x}}\\\cline{1-5}
\lr{\phantom{x}}&\lr{\bullet}&\lr{\phantom{x}}&\lr{\phantom{x}}\\\cline{1-4}
\lr{\phantom{x}}&\lr{\phantom{x}}&\lr{\bullet}&\lr{\phantom{x}}\\\cline{1-4}
\lr{\phantom{x}}&\lr{\phantom{x}}\\\cline{1-2}
\lr{\phantom{x}}\\\cline{1-1}
\end{array}$}
}
\caption{The Ferrers diagram for the partition $(5,4,4,2,1)$}
\end{figure}

The leg length for any diagonal cell is then the number of boxes below it, and the arm length is the number of cells to its right. So the Frobenius coordinates for $\lambda$ are 
\begin{align*}
 \left( \begin{array}{ccc}
    4 & 2 & 1 \\
    4 & 2 & 0\end{array} \right)
    .
\end{align*}

A \emph{generalized Frobenius partition} or \emph{F-partition} is an array as in \eqref{eq:array} but where the rows are allowed to come from more general sets.  In \cite{Andrews}, Andrews introduced this notion and considered two particular types of F-partitions.  In this paper we are concerned with one of these types, namely \emph{the generalized Frobenius partitions in $k$-colors}.  This is obtained by requiring that the sequences $(a_j)_{1\leq j\leq m}$ and $(b_j)_{1\leq j\leq m}$ as in \eqref{eq:array} are each strictly decreasing sequences of integers selected from $k$ copies of the nonnegative integers $\N_0$.  By \emph{strictly decreasing}, we mean with respect to the following lexicographical ordering: if $m$ belongs to the $j$th copy of $\N_0$, which we denote by writing $m=m_j$, and $n=n_\ell$ belongs to the $\ell$th copy of $\N_0$, then $m_j<n_\ell$ precisely when $m<n$ or $m=n$ and $j<\ell$.


Following Andrews' notation, we denote the number of such partitions of $n$ by $c\phi_k(n)$ and define their generating function
\begin{equation}\label{eq:CPhik}
 C\Phi_k(q) := \sum_{n=0}^\infty c\phi_k(n)q^n.
\end{equation}
The special case of $k=1$ corresponds to the usual partition function. 

Given the fact the the partition function satisfies so many striking congruence relations, it is natural to ask whether $c\phi_k(n)$ also has simple congruences.  Indeed, the answer is yes, and there is a long history of results of this type.  In \cite{Andrews} it was shown that $c\phi_p(n)\equiv 0 \pmod{p^2}$ for primes $p\nmid n$.  These general results where extended in \cite{K2} and further in \cite{Sellers}.  The recent paper \cite{GS2014} proved that there exist several infinite families of congruences, for example $c\phi_{5N+1}(5n+4)\equiv 0\pmod{5}$ for any $n,N\in \N$. 

Using explicit realizations of $C\Phi_k$ in terms of $q$-series for small values of $k$, many authors have found additional congruences.  The case $k=3$ was studied in \cite{Kolitsch}, $k=4$ in \cite{BS2011,Lin,Sellers4,Xia}, and $k=6$ in \cite{BS2015,H}. The techniques of these papers are quite similar in that they rely on having an explicit $q$-series representation for $C\Phi_k$.  

The method of Andrews (outlined below) for finding such $q$-series was employed by him in the cases $k=2,3$, and it was generalized to other small values of $k$.  However, this procedure becomes increasingly tedious as $k$ grows.  In this paper we show that by using the language of Jacobi forms and extending the problem to a larger set of functions., the known formulas for the generating functions $C\Phi_k$ can be easily derived.

In addition to giving a robust method for deriving additional formulas, we expect moreover that considering $C\Phi_k$ from the point of view of Jacobi forms provides additional means by which congruences can be studied.  As an example of how ``modularity'' has been used previously, see \cite{Ono} (later extended in \cite{L}), which established congruence properties for $\phi_3(n) \pmod{7}$ by relating $\sum_{n\geq 1}\phi_3(9n)q^n$ to modular forms.

To describe our procedure, we recall the situation for the first few $k$.  Since $c\phi_1$ is the usual partition function, by the well-known formula of Euler, we have 
\begin{equation}\nonumber
C\Phi_1(q)=\frac{1}{\left(q;q\right)_{\infty}},
\end{equation}
where $(a;q)_\infty := \prod_{n=1}^\infty \left(1-aq^{n-1}\right)$.  In Corollary 5.1 of \cite{Andrews} (using a combinatorially defined representation of \eqref{eq:CPhik}), the expressions 
\begin{equation}\label{eq:Andrews2}
 C\Phi_2(q)=\frac{\left(q^2;q^4\right)_\infty}{\left(q;q^2\right)_\infty^4\left(q^4;q^4\right)_\infty}
\end{equation}
and
\begin{equation}\label{eq:Andrews3}
 C\Phi_3(q)
= \frac{\left(q^{12};q^{12}\right)_\infty\left(q^{6};q^{12}\right)_\infty^3}{\left(q;q^6\right)_\infty^5\left(q^5;q^6\right)_\infty^5\left(q^4;q^4\right)_\infty^2\left(q^3;q^6\right)_\infty^7} + 4q\frac{\left(q^{12};q^{12}\right)_\infty\left(q^4;q^4\right)_\infty}{\left(q^6;q^{12}\right)_\infty\left(q^2;q^4\right)_\infty\left(q;q\right)_\infty^3}
\end{equation}
is given.

The proofs of these formulas used the following observation, referred to as the ``General Principle'' in Section~3 of \cite{Andrews}.  That is, $C\Phi_k(q)$ is the constant term (with respect to the $\zeta$ variable) of the infinite product
\begin{equation}\label{eq:Andrewsprod}
(-\zeta q;q)_\infty^k(-\zeta^{-1};q)_\infty^k.
\end{equation}

We may describe this observation in terms of the {\it Jacobi theta function} ($q:=e^{2\pi i \tau}$), 
\begin{equation}\label{eq:theta}
 \vartheta(z;\tau) := \sum_{n\in \frac{1}{2}+\Z}e^{\pi i n^2\tau + 2\pi i n\(z+\frac{1}{2}\)},
\end{equation}
and the {\it Dedekind eta function} 
\begin{equation}\nonumber
  \eta(\tau) := q^{\frac{1}{24}}(q;q)_\infty.
\end{equation}
The Jacobi form of interest is then, for $k\in \N$, 
\begin{equation}\nonumber
F_k(z;\tau) := \left(\frac{-\vartheta\left(z+\frac12;\tau\right)}{q^{\frac1{12}}\eta(\tau)}\right)^k.
\end{equation}
Using \eqref{JTP} below, it is easy to show that the $\frac{k}{2}$th Fourier coefficient of $F_k(z;\tau)$ with respect to the variable $\zeta:=e^{2\pi i z}$ is $C\Phi_k(q)$.

Rather than working just with the constant term of \eqref{eq:Andrewsprod}, our method effectively finds all of the terms in the Fourier expansion (with respect to $z$) of $F_k$.  This is encoded in the so-called theta decomposition of $F_k$.  As described in Section~\ref{sec:prelims}, $F_k(z;\tau)$ has a theta decomposition.  In the case of $k=2\ell$ even, we have
 \[ F_{2\ell}(z;\tau) = \displaystyle{\sum_{a\pmod{2\ell}}H_{\ell,a}(\tau)\vartheta_{\ell,a}(z;\tau)} \]
with $\vartheta_{\ell,a}$ as defined in \eqref{eq:Jthetaam} below.  Our main result gives an iterative formula for the functions $H_{\ell,c}(\tau)$ in terms of $\theta_{m,a}(\tau):=\vartheta_{m,a}(0;\tau)=\sum_{n\in \Z}q^{\frac{(2mn+a)^2}{4m}}$.  This can then be used to calculate $C\Phi_k(q)$.
\begin{theorem}\label{main}
For any $j\in \Z$, let
\begin{equation}\label{eq:base}
  h_{1,j}(\tau):=\theta_{1,j+1}(\tau).
\end{equation}
Given $\{h_{\ell,c}(\tau): 0\leq c < 2\ell\}$, recursively define 
\begin{align}\label{eq:inducttwo}
 \begin{split}
 h_{\ell+1,b}(\tau) & := h_{1,b}(\tau)\theta_{\ell(\ell+1),b\ell}(\tau)h_{\ell,0}(\tau)
       +h_{1,b-\ell}(\tau)\theta_{\ell(\ell+1),b\ell-\ell(\ell+1)}(\tau)h_{\ell,\ell}(\tau) \\
  & \quad + \sum_{c=1}^{\ell-1} h_{1,b-c}(\tau)
    \big( \theta_{\ell(\ell+1),b\ell-c(\ell+1)}(\tau) 
          +\theta_{\ell(\ell+1),b\ell+c(\ell+1)}(\tau) \big) h_{\ell,c}(\tau),
 \end{split}
\\ \label{eq:inductone}
  h_{\ell+\frac12,b+\frac12}(\tau) & := \sum_{c\pmod{2\ell}} h_{\ell,c}(\tau)\theta_{\ell(2\ell+1),c(2\ell+1)-\ell(2b+1)}(\tau),
\end{align}
and for $\ell < b \leq 2\ell$, set $h_{\ell,b}(\tau):=h_{2\ell,2\ell-b}(\tau)$.  Then, for all $\ell\in \N$ and $\delta\in\{0,1\}$, 
\begin{equation}\label{equation2ldelta}
 (-1)^{\delta}\vartheta\Big(z+\frac12;\tau\Big)^{2\ell+\delta} = \sum_{b \pmod{2\ell+\delta}} h_{\ell+\frac{\delta}{2},b+\frac{\delta}{2}}(\tau)\vartheta_{\ell+\frac{\delta}{2},b+\frac{\delta}{2}}(z;\tau). 
\end{equation}
\end{theorem}

\begin{remark}
The recursive formula that gives $h_{\ell+\frac12,b}(\tau)$ in terms of  $h_{\ell,b}(\tau)$ holds even if $\ell\in \frac12\Z$, but one must sum over half integers.  By further making the substitions $\ell\mapsto \ell+1/2$, $b\mapsto b-1/2$, and $c\mapsto c+1/2$, \eqref{eq:inductone} gives the formula
\begin{equation}\nonumber
   h_{\ell+1,b}(\tau) = \sum_{c\pmod{2\ell+1}} h_{\ell+\frac12,c+\frac12}(\tau)\theta_{(2\ell+1)(\ell+1),(2c+1)(\ell+1)-(2\ell+1)b}(\tau)
\end{equation}
which holds for $\ell\in\N$.
\end{remark}

\begin{corollary}\label{thetacor}
All Fourier coefficients of $\vartheta(z+1/2;\tau)$ can be given as combinations  of Dedekind eta functions and Klein forms.
In particular, there exists an algorithm to compute $C\Phi_k(q)$ as a sum of products of $q$-Pochhammer symbols.
\end{corollary}

The outline for the paper as follows.  In Section~\ref{sec:prelims} we give background and definitions for the ideas from the theory of Jacobi forms.  Then, in Section~\ref{sec:proofs} we first show how this theory leads readily to simple derivations of \eqref{eq:Andrews2} and \eqref{eq:Andrews3}, and then we prove Theorem~\ref{main} and Corollary~\ref{thetacor}.  Finally, in Section~\ref{sec:examples} we provide explicit realizations for $C\Phi_k(q)$ for $k=6,7,$ and $8$.

\section{Preliminaries}\label{sec:prelims}

\subsection{Jacobi forms}\label{JacobiDefnSection}
We first require the definition of Jacobi forms, whose theory was laid out in depth by Eichler--Zagier \cite{EZ}. For this, let $\Gamma$ be a congruence subgroup of $\SL_2(\Z)$, and let $\kappa,m\in \Z$.  A \emph{holomorphic Jacobi form of weight $\kappa$ and index $m$ on a congruence subgroup $\Gamma$} is a holomorphic function $\varphi:\C\times \H\to \C$ ($\H$ the complex upper half-plane) which satisfies the following conditions. 
\begin{itemize}
 \item[(i)] For all $\gamma=\abcd\in \Gamma$, 
 \[ \varphi\left(\frac{z}{c\tau+d};\frac{a\tau+b}{c\tau+d}\right)=(c\tau+d)^\kappa e^{\frac{2\pi i mcz^2}{c\tau+d}}\varphi(z;\tau). \]
 \item[(ii)] For all $\lambda,\mu\in \Z$,
 \[ \varphi(z+\lambda\tau+\mu;\tau)=e^{-2\pi i m\left(\lambda^2\tau+2\lambda z\right)}\varphi(z;\tau). \]
 \item[(iii)]  The function $\varphi$ has a Fourier expansion of the form
 \[ \varphi(z;\tau)= \sum_{n,r\in \Z} c(n,r)q^n\zeta^r\]
with $c(n,r)=0$ unless $4mn\geq r^2$.
\end{itemize}
\begin{remark}
As with the ordinary theory of modular forms, there are suitable modifications of this definition to allow for half-integral weight (and index), as well as multiplier systems (as arise in examples such as Jacobi's theta function below). For ease of exposition, we omit these technical definitions here, opting to only give the necessary transformations for the basic Jacobi form needed here, namely the Jacobi theta function.
\end{remark}

\subsection{The Jacobi theta function}\label{sec:JacTheta}
The Jacobi form which is of most importance in this paper is the Jacobi theta function, defined in \eqref{eq:theta}, which is well-known to  satisfy the transformation properties
\begin{equation*}
\vartheta(z;\tau+1) = e^{\frac{\pi i}{4}} \vartheta(z;\tau),\qquad
\vartheta\left(-\textstyle{\frac{z}{\tau};-\frac{1}{\tau}}\right) = i\sqrt{-i\tau}e^{\frac{\pi i z^2}{\tau}}\vartheta(z;\tau),
\end{equation*}
and is an example of a holomorphic Jacobi form of weight $1/2$ and index $1/2$.

Moreover, the Jacobi theta function satisfies the well known Jacobi triple product identity 
\begin{align}\label{JTP}
 \vartheta(z;\tau) =  -iq^{\frac18}\zeta^{-\frac12}\left(q;q\right)_\infty\left(\zeta q^{-1};q\right)_\infty\left(\zeta^{-1};q\right)_\infty.
\end{align}

\subsection{Theta decomposition}\label{sec:thetadecomp} 

The main structural result on Jacobi forms that we exploit in this paper is that every holomorphic Jacobi form can be expressed in terms of the Jacobi theta functions 
\begin{equation}\label{eq:Jthetaam}
 \vartheta_{m,a}(z;\tau) := \sum_{\substack{r\in \Z\\ r\equiv a\pmod{2m}}} q^{\frac{r^2}{4m}}\zeta^r = \sum_{n\in \Z} q^{\frac{(2mn+a)^2}{4m}} \zeta^{2mn+a}.
\end{equation}
In the case that $m$ and $a$ are half integers, which is relevant for this paper, we take the latter sum as the definition.

The following theorem shows that every Jacobi form has a theta decomposition.  

\begin{theorem}[Eichler-Zagier] \label{th:EZ}
Suppose $\varphi(z;\tau)$ is a holomorphic Jacobi form of weight $\kappa$ and index $m$.  Then 
 \[ \varphi(z;\tau) = \sum_{a\pmod{2m}}h_a(\tau)\vartheta_{m,a}(z;\tau), \]
where $(h_a(\tau))_{a\ ({\rm mod}\ 2m)}$ is a vector-valued modular form of weight $\kappa-1/2$.
\end{theorem}

So, in particular, a modified version of Theorem~\ref{th:EZ} covering a slightly more general class of Jacobi functions to which $F_k(z;\tau)$ belongs implies that $F_k(z;\tau)$ has a similar decomposition; Theorem~\ref{main} provides this decomposition directly.  Note that we include Theorem~\ref{th:EZ} as a motivating principle, but not as essential to our proof.  Indeed, Theorem~\ref{main} is deduced directly and independently of Theorem~\ref{th:EZ}.

For future reference, note that the modular forms $\theta_{m,a}(\tau)=\vartheta_{m,a}(0;\tau)$ satisfy the relations
\begin{equation}\label{eq:thetamarelation}
 \theta_{m,2mk\pm a}(\tau)= \theta_{m,a}(\tau),
\qquad \mbox{and}\qquad
 \theta_{m,a}(k\tau)= \theta_{km,ka}(\tau),
\end{equation}
for all $k\in \N$.

\subsection{Klein forms}\label{sec:Klein}
Further distinguished modular forms are obtained by specializing the Jacobi theta function to torsion points and dividing by a power of the $\eta$-function.  To be more precise, for $a\in \mathbb{R}$ we have {\it Klein forms} 
\begin{equation}\label{eq:Klein}
t_{a,0}(\tau):=-q^{\frac{a^2}{2}-\frac{a}{2}+\frac{1}{12}}\frac{\left(q^a;q\right)_\infty\left(q^{1-a};q\right)_\infty}{\left( q;q\right)_\infty^2}.
\end{equation}
Specifically, for $a\in \mathbb{Q}$, these functions are modular forms of weight $-1$, holomorphic in the upper half plane and with possible poles and zeros only at cusps.  (See \cite{KL}.)

\section{Proof of Theorem \ref{main} and Corollary \ref{thetacor}}\label{sec:proofs}

\subsection{First examples of Theorem \ref{main}}\label{sec:firstexamples}
This section is devoted to the cases $k=1,2$, a detailed analysis of which sheds light on the general procedure. The Fourier expansion
\begin{equation}\label{Theta12}
-\vartheta\(z+\frac12;\tau\) = \sum_{n\in \Z} q^{\frac12\(n+\frac12\)^2}\zeta^{n+\frac12}
\end{equation}
is an immediate consequence of \eqref{eq:theta}.  Therefore, 
\begin{equation}\nonumber
 F_1(z;\tau) = \sum_{n\in \Z}\left( \frac{q^{\frac12\(n+\frac12\)^2}}{q^{\frac{1}{12}}\eta(\tau)} \right) \zeta^{n+\frac12},
\end{equation}

We can rederive Andrews' formula for $C\Phi_2(q)$ similarly by employing the following result of \cite{BM2013}, where the authors used the notation $A(z;\tau):=\vartheta(z;\tau)^2/\eta(\tau)^6$.  Accounting for the shift $z\mapsto z+1/2$, they gave the following lemma. 

\begin{lemma}\label{lem:theta2}
The square of $\vartheta(z+\frac12;\tau)$ has the theta decomposition
 \[ \vartheta\left(z+\textstyle{\frac12};\tau\right)^2 = \theta_{1,1}(\tau)\vartheta_{1,0}(z;\tau)+\theta_{1,0}(\tau)\vartheta_{1,1}(z;\tau).\]
\end{lemma}

\begin{proof}
The following proof follows that in \cite{BM2013} but for the reader's convenience, we give it here.  First, using \eqref{Theta12}, note that
 \[  \vartheta\left(z+\textstyle{\frac12};\tau\right)^2 = \sum_{r,s\in\Z}q^{\frac12\(\(r+\frac12\)^2+\(s+\frac12\)^2\)}\zeta^{r+s+1}. \]
Thus, making the change of variables $R:=r+s$ and $S:=r-s$, we obtain
\begin{equation*}
 \vartheta\left(z+\textstyle{\frac12};\tau\right)^2 = \sum_{\substack{R,S\in \Z\\ R\equiv S\pmod{2}}}q^{\frac12\left((R+1)^2+S^2\right)}.
\end{equation*}
Depending on the parity of $R$ and $S$, this sum splits naturally into two parts.  The contribution from $R$ and $S$ even equals 
 \[ \sum_{S\in \Z} q^{S^2}\sum_{R\in \Z}q^{\left(R+\frac12\right)^2}\zeta^{2R+1} = \theta_{1,0}(\tau)\vartheta_{1,1}(z;\tau). \]
The contribution from $R$ and $S$ odd is 
 \[ \sum_{S\in \Z} q^{\left(S+\frac12\right)^2} \sum_{R\in \Z}q^{R^2}\zeta^{2R} = \theta_{1,1}(\tau)\vartheta_{1,0}(z;\tau). \]
Adding these together, the result follows.
\end{proof}

Two well-known consequences of the Jacobi triple product formula are
 \begin{align}\label{eq:theta0eta}
   \theta_{1,0}(\tau) =\ & \frac{\eta(2\tau)^5}{\eta(\tau)^2\eta(4\tau)^2} = \frac{\left(q^2;q^{2}\right)_\infty^5}{\left(q;q\right)_\infty^2\left(q^4;q^{4}\right)_\infty^2},\\
\nonumber
   \theta_{1,1}(\tau) =\ & 2\frac{\eta(4\tau)^2}{\eta(2\tau)} = 2q^{\frac14} \frac{\left(q^{4};q^{4}\right)_\infty^2}{\left(q^{2};q^{2}\right)_\infty}. 
 \end{align}
 In order to find the coefficient of $\zeta$ in $F_2(z;\tau)$, we use \eqref{eq:theta0eta} and Lemma~\ref{lem:theta2}.  This leads to
\begin{align*}
 C\Phi_2(q) = \frac{q^{\frac14}\theta_{1,0}(\tau)}{q^{\frac16}\eta(\tau)^2}\ & = \frac{\left(q^2;q^{2}\right)_\infty^5}{\left(q;q\right)_\infty^4\left(q^4;q^{4}\right)_\infty^2},
\end{align*}
which agrees with \eqref{eq:Andrews2}.

\subsection{Preliminary results}

We next prove two lemmas, both of which can be viewed as variations on Lemma~\ref{lem:theta2}.  These lemmas describe how one can obtain the theta decomposition of the weight and index $1$ Jacobi form (with multiplier) 
\begin{equation}\label{eq:fortwomorelem}
 \vartheta_{1,\eps}(z;\tau)\vartheta_{\ell,c}(z;\tau) 
\end{equation}
for $\eps\in\{0,1\}$, and for 
\begin{equation}\nonumber
 \vartheta\left( z+\frac12;\tau \right) \vartheta_{\ell,c}(z;\tau),
\end{equation}
respectively.  

With Lemma~\ref{lem:theta2} in hand, knowing the theta decomposition of \eqref{eq:fortwomorelem} with $\ell=1$ and $c\in\{0,1\}$ one can find the theta decomposition of $\vartheta(z+\frac12;\tau)^3$.  More generally, for any $\ell\in\N$, Lemma~\ref{lem:theta2} provides the key step to go from the theta decomposition of $\vartheta(z+\frac12;\tau)^\ell$ to that of $\vartheta(z+\frac12;\tau)^{\ell+2}$.

\begin{lemma}\label{lem:onemore}
Assume the notation as above.  Then, for $c,\ell\in \N_0$ with $\ell>c $, 
 \[ -\vartheta\left( z+\frac12;\tau \right) \vartheta_{\ell,c}(z;\tau) 
    = \sum_{a\pmod{2\ell+1}} \theta_{\ell(2\ell+1),c-2\ell a-\ell}(\tau)\vartheta_{\ell+\frac12,a+c+\frac12}(z;\tau). \]
\end{lemma}

\begin{proof}
As in the proof of Lemma~\ref{lem:theta2}, we see that
\begin{align}\nonumber
 -\vartheta\left( z+\frac12;\tau \right) \vartheta_{\ell,c}(z;\tau) 
  = & \sum_{r,s\in \Z} q^{ \frac12\left( s+\frac12\right)^2+\frac{(2\ell r+c)^2}{4\ell}} \zeta^{2\ell r+s+c+\frac12}.
\end{align}
We make the change of variables $r= \frac{R+S}{2\ell+1}, s = \frac{S-2\ell R}{2\ell+1}$. As $(r,s)$ runs through $\Z^2$, $(R,S)$ run through those elements in $\Z^2$ satisfying, as claimed, $R+S\equiv 0\pmod{2\ell+1}$.  Thus we find 
\begin{align*}
 -\vartheta&\left( z+\frac12;\tau \right) \vartheta_{\ell,c}(z;\tau) \\
  = & \sum_{\substack{R,S\in \Z \\ R+S\equiv 0 \pmod{2\ell+1}}} q^{ \frac{\ell}{2\ell+1}\left( R+\frac{c}{2\ell}-\frac12 \right)^2 + \frac{1}{2(2\ell+1)}\left( S+c+\frac12 \right)^2 }\zeta^{S+c+\frac12} \\
  = & \sum_{a\pmod{2\ell+1}} \left(
 \sum_{R\in \Z} q^{\frac{\ell}{2\ell+1}\left((2\ell+1)R-a+\frac{c}{2\ell}-\frac12\right)^2} 
\sum_{S\in\Z} q^{\frac{1}{2(2\ell+1)}\left( (2\ell+1)S+a+c+\frac12 \right)^2} \zeta^{(2\ell+1)S+a+c+\frac12} \right).
\end{align*}
From this, we conclude the claim.
\end{proof}

\begin{lemma}\label{lem:theta1eps}
With the same conditions as in Lemma \ref{lem:onemore}, we have that
 \[ \vartheta_{1,\eps}(z;\tau)\vartheta_{\ell,c}(z;\tau) = \sum_{a \pmod{\ell+1}}\theta_{\ell(\ell+1),(2a+\eps)\ell-c}(\tau)\vartheta_{\ell+1,2a+c+\eps}(z;\tau). \]
\end{lemma}

The proof of Lemma~\ref{lem:theta1eps} follows the same idea as that of Lemma \ref{lem:onemore}.  We leave it as an exercise for the interested reader.

\begin{remark}
After proving Lemma~\ref{lem:theta1eps}, Bruce Berndt informed the authors of a very similar result, known as Schroeter's formula, which gives the identity
 \[ T\left(x,q^a\right)T\left(x,q^b\right) = \sum_{k=0}^{a+b-1}y^kq^{bk^2}T\left(xyq^{2bk},q^{a+b}\right)T\left(y^ax^{-b}q^{2abk},q^{ab(a+b)}\right), \]
where $T(x,q):=\sum_{n\in \Z}x^n q^{n^2}$.  We were also made aware that the main theorem of \cite{Cao} is a generalization of Schroeter's theorem and that Lemmas~\ref{lem:theta2}--\ref{lem:theta1eps} should follow as well.
\end{remark}

\subsection{Proof of Theorem \ref{main}}

We are now ready to prove Theorem~\ref{main}.  Let $\ell\in \N$ and $\delta\in \{0,1\}$.  Our method of proof is to give an iterative procedure for obtaining $\vartheta(z+\frac12;\tau)^{2\ell+1+\delta}$ from $\vartheta(z+\frac12;\tau)^{2\ell}$.  Suppose that a decomposition
\begin{equation}\label{equation2l}
 \vartheta\Big(z+\frac12;\tau\Big)^{2\ell} = \sum_{c \pmod{2\ell}} h_{\ell,c}(\tau)\vartheta_{\ell,c}(z;\tau)
\end{equation}
with $h_{\ell,c}=h_{\ell,2\ell-c}$ is known to hold for $\vartheta(z+\frac12;\tau)^{2\ell}$.   By Lemma~\ref{lem:theta2}, \eqref{equation2l} is true in the case $\ell=1$.  This establishes \eqref{eq:base}.  We must then show that defining $h_{\ell+\frac{1-\delta}{2},b+\frac{\delta}{2}}$ as in \eqref{eq:inducttwo} and \eqref{eq:inductone}, it is also true that \eqref{equation2ldelta} holds.

We first treat the case $\delta=0$.  In this case, applying Lemmas~\ref{lem:theta2} and \ref{lem:theta1eps}, we find that 
\begin{align*}
 \vartheta &\Big(z+\frac12;\tau\Big)^{2(\ell+1)}
   =  \Big(\theta_{1,1}(\tau) \vartheta_{1,0}(z;\tau)+\theta_{1,0}(\tau)\vartheta_{1,1}(z;\tau)\Big) \vartheta\Big(z+\frac12;\tau\Big)^{2\ell} \\
    &=  \sum_{c\pmod{2\ell}}h_{\ell,c}(\tau)
    \left( \theta_{1,1}(\tau)\sum_{a\pmod{\ell+1}} \theta_{\ell(\ell+1),2a\ell-c}(\tau)\vartheta_{\ell+1,2a+c}(z;\tau)\right. \\
   &\qquad \qquad \qquad \qquad \quad
  + \left. \theta_{1,0}(\tau)\sum_{\alpha\pmod{\ell+1}} \theta_{\ell(\ell+1),(2\alpha+1)\ell-c}(\tau)\vartheta_{\ell+1,2\alpha+1+c}(z;\tau)\right).
\end{align*}

For a given $b\pmod{2\ell+2}$, we now collect all of the terms in which $\vartheta_{\ell+1,b}(z;\tau)$ appears.  In the sum over $a$, there is exactly one such term if $c\equiv b\pmod{2}$ and none otherwise.  This term comes from the unique $a$ such that $2a+c\equiv b \pmod{2\ell+2}$.  Similarly, in the sum over $\alpha$, we get one term exactly if $c\nequiv b\pmod{2}$, namely for the unique $\alpha$ for which $2\alpha+c+1\equiv b\pmod{2\ell+2}$.  Hence,
\begin{align*}
 h_{\ell+1,b}(\tau) = &
\sum_{\substack{c\pmod{2\ell}\\ c\equiv b \pmod{2}}} 
  \theta_{1,1}(\tau) \theta_{\ell(\ell+1),(b-c)\ell-c}(\tau)h_{\ell,c}(\tau) \\
 & \quad +
\sum_{\substack{c\pmod{2\ell}\\ c\nequiv b \pmod{2}}}
  \theta_{1,0}(\tau) \theta_{\ell(\ell+1),(b-c)\ell-c}(\tau)h_{\ell,c}(\tau). 
\end{align*}
Since $h_{1,b-c}$ is equal to $\theta_{1,1}$ or $\theta_{1,0}$ exactly depending on whether $b$ and $c$ have the same parity or not, this can be simplified to
\begin{equation}\label{eq:inductalmost}
 h_{\ell+1,b}(\tau) = \sum_{c\pmod{2\ell}}
  h_{1,b-c}(\tau) \theta_{\ell(\ell+1),b\ell-c(\ell+1)}(\tau)h_{\ell,c}(\tau).
\end{equation}
From this formula and the inductive hypothesis that $h_{\ell,b}=h_{\ell,2\ell-b}$, it follows that $h_{\ell+1,b}=h_{\ell+1,2\ell+2-b}$.  Using this and the fact that 
$\theta_{\ell(\ell+1),a}$ depends only on $a\pmod{2\ell(\ell+1}$, \eqref{eq:inductalmost} now simplifies readily to \eqref{eq:inducttwo}.

The case of $\delta=1$ corresponds to \eqref{eq:inductone}.  This case follows the same logic as above with the only difference being to employ Lemma~\ref{lem:onemore} instead of Lemma~\ref{lem:theta1eps}.  We leave the details to the reader.

\subsection{Proof of Corollary \ref{thetacor}}\label{sec:proofofCor}

Corollary \ref{thetacor} is a direct consequence of the fact that the function $\theta_{m,b}$ can be expressed as an infinite product in terms of $q$-Pochhammer symbols. Indeed, from \eqref{eq:theta}, \eqref{JTP}, and \eqref{eq:Klein}, it follows that
\begin{align}
 \theta_{m,b}(\tau) = & q^{\frac{b^2}{4m}}\left(q^{2m};q^{2m}\right)_\infty\left(q^{m-b};q^{2m}\right)_\infty\left(q^{m+b};q^{2m}\right)_\infty \nonumber \\
 = & -q^{\frac{m}{12}}\left( q^{2m};q^{2m}\right)_\infty^3 t_{\frac{1}{2}+\frac{b}{2m},0}(2m\tau). \label{eq:thetambasprod}
\end{align}
Up to a multiple by a power of $q$ and a power of $\eta$, $C\Phi_k(q)$ is $h_{\frac{k}{2},\frac{k}{2}}$.  Since Theorem~\ref{main} expresses  $h_{\frac{k}{2},\frac{k}{2}}$ as a combination of the functions $\theta_{m,b}$, the corollary is evident.

\section{Examples}\label{sec:examples}

In this section, whenever we are only dealing with functions of the single variable $\tau$, or equivalently $q$, in order to give cleaner formulas we drop the arguments of the function if possible except in the statements of the theorems.

\subsection{The case of $C\Phi_6$}\label{sec:cphi6}
The identities of \eqref{eq:thetamarelation} are used repeatedly in the following formulas.

Applying \eqref{eq:inducttwo} of Theorem~\ref{main} twice, we find first that 
\begin{equation}\label{eq:h2b}
  h_{2,0} = \theta_{1,1}^2\theta_{2,0}+\theta_{1,0}^2\theta_{2,2}, \quad
  h_{2,1} = 2\theta_{1,0}\theta_{1,1}\theta_{2,1}, \quad
  h_{2,2} = \theta_{1,1}^2\theta_{2,2}+\theta_{1,0}^2\theta_{2,0},
\end{equation}
and then that 
\begin{align}
  h_{3,0}  
  = &\ \theta_{1,1}\theta_{6,0}h_{2,0}+2\theta_{1,0}\theta_{6,3}h_{2,1}
   + \theta_{1,1}\theta_{6,6}h_{2,2}, 
\\ 
  h_{3,1}
  = &\ \theta_{1,0}\theta_{6,2}h_{2,0}+\theta_{1,1}(\theta_{6,1}
   + \theta_{6,5})h_{2,1}+\theta_{1,0}\theta_{6,4}h_{2,2} = h_{3,5},
\\
  h_{3,2}  
  = &\ \theta_{1,1}\theta_{6,4}h_{2,0}+\theta_{1,0}(\theta_{6,1}+\theta_{6,5})h_{2,1}+\theta_{1,1}\theta_{6,2}h_{2,2} = h_{3,4}, 
\\ 
  h_{3,3}  
  = &\ \theta_{1,0}\theta_{6,6}h_{2,0}+2\theta_{1,1}\theta_{6,3}h_{2,1}+\theta_{1,0}\theta_{6,0}h_{2,2}.\label{eq:h33}
\end{align}

From this expression, we can now give a formula for $C\Phi_6(q)$.
\begin{proposition}\label{prop:CPhi6}
The generating function $C\Phi_6(q)$ is equal to 
\begin{equation*} 
\frac{q^{\frac14}}{\eta(\tau)^6} \bigg(
   6\theta_{1,1}(\tau)^2\theta_{1,0}(\tau)\theta_{3,1}(3\tau)\theta_{2,1}(\tau) 
   + \theta_{1,0}(\tau)^3\big( \theta_{1,1}(6\tau)\theta_{1,1}(2\tau)+\theta_{1,0}(6\tau)\theta_{1,0}(2\tau) \big)
 \bigg) 
.
\end{equation*}
\end{proposition}

\begin{remark}
Note that Proposition~\ref{prop:CPhi6} agrees with the result of \cite{BS2015}.  Indeed, using the identities
\begin{gather}
 \varphi(q) = \theta_{1,0}(\tau), \quad
 \qquad 2q^{\frac14}\psi\left(q^2\right) = \theta_{1,1}(\tau), \label{eq:BS} \\
 4q\psi(q)^3\psi\left(q^2\right)\psi\left(q^3\right) = \theta_{1,1}(\tau)\theta_{1,0}(\tau)\theta_{2,1}(\tau)\theta_{2,1}(3\tau), \label{eq:BS3}
\end{gather}
where $\varphi(q):=\sum_{n\in\Z}q^{n^2}$ and $\psi(q):=\sum_{n=0}^\infty q^{\frac{n(n+1)}{2}}$ is the notation used in \cite{BS2015}, Proposition~\ref{prop:CPhi6} is equivalent to Theorem~2.1 of \cite{BS2015}.  The identitities in \eqref{eq:BS} follow readily from the definitions, and \eqref{eq:BS3} can be checked using Sturm's Theorem.
\end{remark}

The following lemma is used in the proofs of Proposition~\ref{prop:CPhi6} and Proposition~\ref{prop:CPhi7}.  Using Sturm's Theorem, its proof is immediate, however, we include an elementary direct proof with the hope that it could be generalized.  Such generalizations may allow one to simplify or even give a closed form for $h_{k,k}$ for arbitrary $k\in \frac12 \N_0$.

\begin{lemma}\label{lem:Theta3simplify}
We have 
\begin{align}
 \theta_{2,2}(\tau)\theta_{6,0}(\tau)+\theta_{2,0}(\tau)\theta_{6,6}(\tau)=&\ 2\theta_{2,1}(\tau)\theta_{6,3}(\tau), \label{eq:21lemA} \\
 \theta_{2,2}(\tau)\theta_{6,4}(\tau)+\theta_{2,0}(\tau)\theta_{6,2}(\tau)=&\ \theta_{2,1}(\tau)(\theta_{6,1}(\tau)+\theta_{6,5}(\tau)). \label{eq:21lemB}
\end{align}
\end{lemma}
 
\begin{proof}
Using the definitions and simplifying, we find that
 \[ \theta_{2,2}(\tau)\theta_{6,0}(\tau)+\theta_{2,0}(\tau)\theta_{6,6}(\tau)
 = \sum_{\substack{R,S\in \Z\\R\nequiv S \pmod{2}}} \(q^{\frac14}\)^{2R^2+6S^2}. \]

We now make the change of variables $s = \frac{R+S}{2},\ r = \frac{R-3S}{2}$.
Since $R$ and $S$ have opposite parity, $r,s\in \frac12+\Z$.  To ease the notation we also replace $q^{\frac14}$ with $q$.  Applying this change of variables gives 
\begin{equation*}
 \sum_{\substack{R,S\in \Z\\R\nequiv S \pmod{2}}} q^{2R^2+6S^2}
  = \sum_{\substack{r,s\in \frac12+\Z\\r\equiv s \pmod{2}}} q^{2r^2+ 6s^2}.
\end{equation*}  
We now write $r=1/2+j$ and $s=1/2+k$.  Since the condition $r\equiv s\pmod{2}$ is equivalent to $j\equiv k\pmod{2}$, it follows that this is equal to
\begin{align*}
 \sum_{\substack{j,k\in \Z\\j\equiv k \pmod{2}}} q^{2\(j+\frac12\)^2+6\(k+\frac12\)^2} 
  & = \sum_{j,k\in \Z} \(q^4\)^{2\(j+\frac14\)^2+6\(k+\frac14\)^2}
   + \sum_{j,k\in \Z} \(q^4\)^{2\(j-\frac14\)^2+6\(k-\frac14\)^2}.
\end{align*}
It is easy to see, replacing $q^4$ with $q$, that each of these final two sums is equal to $\theta_{2,1}(3\tau)\theta_{2,1}(\tau)$.  This proves \eqref{eq:21lemA}.  The identity \eqref{eq:21lemB} is similar; we omit the details.
\end{proof}

\begin{remark}
The terms 
\[ \theta_{2,0}\theta_{6,0}+\theta_{2,2}\theta_{6,2}\qquad\mbox{and}\qquad
    \theta_{2,0}\theta_{6,4}+\theta_{2,2}\theta_{6,2} \]
appear in the formulas for $C\Phi_6$ and $C\Phi_7$, and so it is natural to ask whether relations similar to \eqref{eq:21lemA} and \eqref{eq:21lemB} exist for these.  We do not, however, believe that such simple relations hold in these cases.
\end{remark}

\begin{proof}[Proof of Proposition~\ref{prop:CPhi6}]
From Theorem~\ref{main}, we see that
\begin{align}\label{eq:CPhi6asCoeff}
 &C\Phi_6(q)=\mathrm{coeff}_{\left[\zeta^3\right]}\( \frac{\vartheta(z+\frac12;\tau)^6}{q^{\frac12}\eta(\tau)^6} \)
 =\frac{ \mathrm{coeff}_{\left[\zeta^3\right]}\vartheta_{3,3}(z;\tau)}{ q^{\frac12}\eta(\tau)^6 } h_{3,3}(\tau).
\end{align}
In order to obtain a more explicit formula for $C\Phi_6$, we plug the results of \eqref{eq:h2b} into \eqref{eq:h33}.  This yields
\begin{align}
 h_{3,3} &= 4\theta_{1,0}\theta_{1,1}^2\theta_{2,1}\theta_{6,3}
    + \theta_{1,0} \theta_{1,1}^2\big( \theta_{2,0}\theta_{6,6}+\theta_{2,2}\theta_{6,0}\big)
    + \theta_{1,0}^3 \big( \theta_{2,2}\theta_{6,6}+\theta_{2,0}\theta_{6,0}\big)  \nonumber
\\
 & = 6\theta_{1,0}\theta_{1,1}^2\theta_{2,1}\theta_{6,3}
    + \theta_{1,0}^3 \big( \theta_{2,2}\theta_{6,6}+\theta_{2,0}\theta_{6,0}\big), \label{eq:CPhi6Final}
\end{align}
using Lemma~\ref{lem:Theta3simplify} in the last step.  Since $\mathrm{coeff}_{[\zeta^3]}\vartheta_{3,3}(z;\tau)=q^{\frac34}$, the desired result follows from \eqref{eq:CPhi6asCoeff} and \eqref{eq:CPhi6Final}.
\end{proof}

\subsection{The case of $C\Phi_7$}

From Theorem~\ref{main} and \eqref{eq:thetamarelation}, after some simplification, we see that
\begin{align*}
 h_{\frac72,\frac72} = & \sum_{c\pmod{6}} h_{3,c} \theta_{21,7c-21} 
 = h_{3,0} \theta_{21,21} + 2h_{3,1}\theta_{21,14} + 2h_{3,2}\theta_{21,7} + h_{3,3}\theta_{21,0}.
\end{align*}
Using the expression for $h_{3,c}$ and $h_{2,c}$ given in equations~\eqref{eq:h2b}--\eqref{eq:h33}, we can write this completely in terms of $\theta_{p,q}$.  Upon simplifying, $h_{\frac72,\frac72}$ is equal to
\begin{align}
 &  \Big(\theta_{1,1}\theta_{6,0}\theta_{21,21} 
   + 2\theta_{1,0}\theta_{6,2}\theta_{21,14}
   + 2\theta_{1,1}\theta_{6,4}\theta_{21,7}
   + \theta_{1,0}\theta_{6,6}\theta_{21,0}\Big)
    \Big( \theta_{1,1}^2\theta_{2,0}+\theta_{1,0}^2\theta_{2,2} \Big) \nonumber \\
 & \quad +  
  4\theta_{1,0}\theta_{1,1}\theta_{2,1}
   \Big(\theta_{6,3}(\theta_{1,0}\theta_{21,21}+\theta_{1,1}\theta_{21,0})
   + (\theta_{6,1} + \theta_{6,5})
     (\theta_{1,1}\theta_{21,14}+\theta_{1,0}\theta_{21,7})\Big) \nonumber \\
 & \quad +  
  \Big(\theta_{1,1}\theta_{6,6}\theta_{21,21}
   + 2\theta_{1,0}\theta_{6,4}\theta_{21,14}
   + 2\theta_{1,1}\theta_{6,2}\theta_{21,7}
   + \theta_{1,0}\theta_{6,0}\theta_{21,0}\Big)
    \Big( \theta_{1,1}^2\theta_{2,2}+\theta_{1,0}^2\theta_{2,0} \Big). \label{eq:h7272}
\end{align}

With this in hand, we deduce the following result.

\begin{proposition}\label{prop:CPhi7}
We have $C\Phi_7(q)=h_{\frac72,\frac72}(\tau)/(q;q)_\infty^7$ with $h_{\frac72,\frac72}$ equal to
\footnotesize
 \begin{align*}
 &
6\theta_{1,0}\theta_{1,1}\theta_{2,1}
    \Big(\theta_{6,3}\big(\theta_{1,0}\theta_{21,21}+\theta_{1,1}\theta_{21,0}\big)
   + \big(\theta_{1,1}\theta_{21,14} + \theta_{1,0}\theta_{21,7}\big)
     \big(\theta_{6,1} + \theta_{6,5}\big)\Big) \\
 & \quad +
\Big(\theta_{1,0}^3\theta_{21,0}+\theta_{1,1}^3\theta_{21,21}\Big)
   \Big(\theta_{2,0}\theta_{6,0}+\theta_{2,2}\theta_{6,6} \Big) 
    + 2\Big( \theta_{1,0}^3\theta_{21,14} + \theta_{1,1}^3\theta_{21,7} \Big) 
   \Big(\theta_{2,0}\theta_{6,4}+\theta_{2,2}\theta_{6,2} \Big).
 \end{align*}
\normalsize
\end{proposition}
\begin{proof}
Note that $\mathrm{coeff}_{[\zeta^{7/2}]}\vartheta(z;\tau)=q^{7/8}$.  Thus, in complete analogy to \eqref{eq:CPhi6asCoeff}, we have 
 \[ C\Phi_7(q)= q^{\frac{7}{24}} \frac{h_{\frac72,\frac72}(\tau)}{\eta(\tau)^7} 
  = \frac{h_{\frac72,\frac72}(\tau)}{(q;q)_\infty^7}, \]
with $h_{\frac72,\frac72}$ as in \eqref{eq:h7272}.  In order to further simplify this expression, we first consider only the terms coming from the first and third lines of \eqref{eq:h7272}.  After rearranging and applying \eqref{eq:21lemA} and \eqref{eq:21lemB}, it can be shown that they are equal to
\footnotesize
\begin{align*}
& 2\theta_{1,0}\theta_{1,1}\theta_{2,1}\theta_{6,3}\Big( \theta_{1,0}\theta_{21,21}+\theta_{1,1}\theta_{21,0}\Big) 
   + \Big(\theta_{1,0}^3\theta_{21,0}+\theta_{1,1}^3\theta_{21,21}\Big)
   \Big(\theta_{2,0}\theta_{6,0}+\theta_{2,2}\theta_{6,6} \Big) \\
 & + 2\Big( \theta_{1,0}^3\theta_{21,14} + \theta_{1,1}^3\theta_{21,7} \Big) 
   \Big(\theta_{2,0}\theta_{6,4}+\theta_{2,2}\theta_{6,2} \Big) +
   2\theta_{1,0}\theta_{1,1}\theta_{2,1}\Big(\theta_{6,1}+\theta_{6,5} \Big)\Big(\theta_{1,0}\theta_{21,7}+\theta_{1,1}\theta_{21,14}\Big).
\end{align*}
\normalsize
Plugging these back into \eqref{eq:h7272} and simplifying gives the desired formula for $h_{\frac72,\frac72}$.
\end{proof}

At this point, one could give $C\Phi_7(q)$ in terms of $\eta$ and Klein forms using \eqref{eq:thetambasprod}.  Such a formula would be quite long, however, so we refrain from writing it down explicitly.

\subsection{The case of $C\Phi_8$}
Repeating the method of the previous section, but using \eqref{eq:inducttwo}, we find that
\begin{align*}
 h_{4,4} = &\ h_{1,4}\theta_{12,12}h_{3,0} + h_{1,1}\theta_{12,0}h_{3,3} 
 + \sum_{c=1}^2 h_{1,4-c}\big( \theta_{12,12-4c} + \theta_{12,12+4c}\big)h_{3,c} \\
 = &\ \theta_{1,1} \theta_{12,12}h_{3,0} + 2\theta_{1,0}\theta_{12,8}h_{3,1} 
 + 2\theta_{1,1}\theta_{12,4}h_{3,2} + \theta_{1,0}\theta_{12,0}h_{3,3}.
\end{align*}
Now we insert the formulas for $h_{3,c}$ and $h_{2,c}$ as above to find that $h_{4,4}$ is equal to
\begin{align*}
 & \Big(\theta_{1,1}^2\theta_{12,12}\theta_{6,0} + 2\theta_{1,0}^2\theta_{12,8}\theta_{6,2}+2\theta_{1,1}^2\theta_{12,4}\theta_{6,4} + 2\theta_{1,0}^2\theta_{12,0}\theta_{6,6}\Big)\Big( \theta_{1,1}^2\theta_{2,0}+\theta_{1,0}^2\theta_{2,2} \Big) \\
 & \quad + 4\theta_{1,0}^2\theta_{1,1}^2\theta_{2,1}\Big( (\theta_{12,0}+\theta_{12,12})\theta_{6,3} + (\theta_{12,8}+\theta_{12,4})(\theta_{6,1}+\theta_{6,5}) \Big) \\
 & \quad + \Big(\theta_{1,1}^2\theta_{12,12}\theta_{6,6} + 2\theta_{1,0}^2\theta_{12,8}\theta_{6,4} + 2\theta_{1,1}^2\theta_{12,4}\theta_{6,2} + \theta_{1,0}^2\theta_{12,0}\theta_{6,0}\Big)\Big( \theta_{1,1}^2\theta_{2,2}+\theta_{1,0}^2\theta_{2,0} \Big).
\end{align*}


\begin{thebibliography}{10}
\bibitem{Andrews}
G.E.~Andrews, \emph{Generalized {F}robenius partitions}, Mem. Amer. Math.
  Soc. \textbf{49} (1984), no.~301, iv+44.

\bibitem{BS2011}
N.D.~Baruah and B.K.~Sarmah, \emph{Congruences for generalized
  {F}robenius partitions with 4 colors}, Discrete Math. \textbf{311} (2011),
  no.~17, 1892--1902.

\bibitem{BS2015}
\bysame, \emph{Generalized {F}robenius partitions with 6 colors}, Ramanujan J.
  \textbf{38} (2015), no.~2, 361--382.

\bibitem{BM2013}
K.~Bringmann and S.~Murthy, \emph{On the positivity of black hole
  degeneracies in string theory}, Commun. Number Theory Phys. \textbf{7}
  (2013), no.~1, 15--56.

\bibitem{Cao}
Z.~Cao, \emph{Integer matrix exact covering systems and product identities for theta functions}, Int. Math. Res. Not. IMRN (2011), no.~19, 4471--4514.

\bibitem{EZ}
M.~Eichler and D.~Zagier, \emph{The theory of {J}acobi forms}, Progress in
  Mathematics, vol.~55, Birkh\"auser Boston, Inc., Boston, MA, 1985.

\bibitem{GS2014}
F.G.~Garvan and J.A.~Sellers, \emph{Congruences for generalized
  {F}robenius partitions with an arbitrarily large number of colors}, Integers
  \textbf{14} (2014), Paper No. A7, 5.

\bibitem{H}
M.D.~Hirschhorn, \emph{Some congruences for 6-colored generalized
  {F}robenius partitions}, Ramanujan J. \textbf{40} (2016), no.~3, 463--471.

\bibitem{Kolitsch}
L.W.~Kolitsch, \emph{A congruence for generalized {F}robenius partitions
  with {$3$} colors modulo powers of {$3$}}, Analytic number theory ({A}llerton
  {P}ark, {IL}, 1989), Progr. Math., vol.~85, Birkh\"auser Boston, Boston, MA,
  1990, pp.~343--348.

\bibitem{K2}
\bysame, \emph{An extension of a congruence by {A}ndrews for
  generalized {F}robenius partitions}, J. Combin. Theory Ser. A \textbf{45}
  (1987), no.~1, 31--39.

\bibitem{KL}
D.S.~Kubert and S.~Lang, \emph{Modular units}, Grundlehren der
  Mathematischen Wissenschaften [Fundamental Principles of Mathematical
  Science], vol. 244, Springer-Verlag, New York-Berlin, 1981.

\bibitem{Lin}
B.L.S.~Lin, \emph{New {R}amanujan type congruence modulo 7 for 4-colored
  generalized {F}robenius partitions}, Int. J. Number Theory \textbf{10}
  (2014), no.~3, 637--639.

\bibitem{L}
J.~Lovejoy, \emph{Ramanujan-type congruences for three colored {F}robenius
  partitions}, J. Number Theory \textbf{85} (2000), no.~2, 283--290.

\bibitem{Ono}
K.~Ono, \emph{Congruences for {F}robenius partitions}, J. Number Theory
  \textbf{57} (1996), no.~1, 170--180.

\bibitem{Sellers}
J.A.~Sellers, \emph{Congruences involving generalized {F}robenius partitions},
  Internat. J. Math. Math. Sci. \textbf{16} (1993), no.~2, 413--415.

\bibitem{Sellers4}
\bysame, \emph{An unexpected congruence modulo 5 for 4-colored
  generalized {F}robenius partitions}, J. Indian Math. Soc. (N.S.) (2013),
  no.~Special volume to commemorate the 125th birth anniversary of Srinivasa
  Ramanujan, 97--103.

\bibitem{Xia}
E.X.W.~Xia, \emph{Proof of a conjecture of {B}aruah and {S}armah on
  generalized {F}robenius partitions with 6 colors}, J. Number Theory
  \textbf{147} (2015), 852--860.

\end{thebibliography}
\end{document}